\documentclass[12pt]{amsart}
\newcommand\Cuntz{\mathbf S_m} 
\newcommand\A{\mathcal A}  
\newcommand\Z{\mathbb Z} 
\newcommand\C{\mathbb C} 
\newcommand\F{\mathbf F} 
\newcommand\FS{\mathbf FS_m} 

\newcommand\Cc{{\mathcal C}} 

\newcommand\HH{{\mathcal H\mathcal H}} 
\newcommand\CC{\mathcal C\mathcal C} 
\newcommand\HC{\mathcal H\mathcal C} 
\newcommand\BC{\mathcal B\mathcal C} 
\newcommand\ZC{\mathcal Z\mathcal C} 

\newcommand{\SB}[2]{{\mathcal B}^{#1}({#2},{#2}')} 
\newcommand{\SC}[2]{{\mathcal C}^{#1}({#2},{#2}')} 
\newcommand{\SZ}[2]{{\mathcal Z}^{#1}({#2},{#2}')} 

\renewcommand\d{\delta}

\newcommand{\ptp}{\widehat{\otimes}}  
\newcommand\oh{\hat\otimes}
\newcommand{\et}{{\bf x}}  
\newcommand{\chain}{a_1 \otimes \cdots \otimes a_{n+1}}

\newcommand{\subchain}[2]{a_{#1} \otimes \cdots \otimes a_{#2}}
\newcommand{\achain}{a_1 \otimes \cdots \otimes a_{n+1}}
\newcommand{\asubchain}[2]{a_{#1} \otimes \cdots \otimes a_{#2}}

\newcommand{\defeq}{:=}

\def\AC{{\mathcal{A}_S}} 
\newcommand\order{{\rm \bf l}}
\def\one{\hbox{$\mathbf 1$}}
\def\zero{\hbox{$z_0$}}

\def\palpha{p_{\alpha}}
\def\qbeta{q_{\beta}}

\def\pa#1{p_{\alpha_{#1}}}
\def\qb#1{q_{\beta_{#1}}}

\def\p#1{p_{\alpha_{#1}}}
\def\q#1{q_{\beta_{#1}}}

\def\ds#1#2{d^n_{#1}s^n_{#2}}
\def\sd#1#2{s^{n-1}_{#1}d^{n-1}_{#2}}
\def\sdplusds{s^{n-1}d^{n-1}+d^n s^n}

\def\stwo{r}
\def\sdplusdstwo{\stwo^{n-1}d^{n-1}+d^n \stwo^n}

\def\sfg{s_{V}} 
\def\rhov{\bar{\rho}} 
\def\etv{{\mathbf \omega}} 
\newcommand{\subvchain}[2]{\etv_{#1} \otimes \cdots \otimes \etv_{#2}}
\def\sdplusdsv{\sfg^{n-1}d^{n-1}+d^n \sfg^n}

\newcommand{\dt}[1]{{\sf #1}}
\newcommand{\dif}{\delta}

\newcommand{\tp}{{\mathop{\scriptstyle{\otimes}}}}
\renewcommand{\tp}{\otimes}

\newcommand{\cyc}{\operatorname{\bf t}} 

\newcommand{\Co}[3][]{{\mathcal {#2}_{#1}^{#3}}}
\newcommand{\Ho}[3][]{{\mathcal {#2}^{#1}_{#3}}}

\title[The cyclic and simplicial cohomology of the Cuntz semigroup algebra]
{Cyclic and simplicial cohomology of the Cuntz semigroup algebra}

\author{Fr\'ed\'eric Gourdeau}

\author{Michael C. White}

\date{January 28, 2010}

\newtheorem{Proposition}{Proposition}[section]
\newtheorem{Theorem}[Proposition]{Theorem}
\newtheorem{Lemma}[Proposition]{Lemma}

\newtheorem{Definition}[Proposition]{Definition}

\theoremstyle{remark}
\newtheorem{remark}[Proposition]{Remark}

\begin{document}
\baselineskip = 17pt

  \begin{abstract}
The main objective of this paper is to determine the simplicial and cyclic cohomology groups of the Cuntz semigroup algebra $\ell^1(\Cuntz)$. In order to do so, we first establish some general results which can be used when studying simplicial and cyclic cohomology of Banach algebras in general.
We then turn our attention to $\ell^1(\Cuntz)$, showing that the cyclic cohomology groups of degree $n$ vanish when $n$ is odd and are one-dimensional when $n$ is even ($n\ge 2$). Using the Connes-Tzygan exact sequence, these results are used to show that the simplicial cohomology groups of degree $n$ vanish for $n\ge 1$. We also determine the simplicial and cyclic cohomology of the tensor algebra of a Banach space, a class which includes the algebra on the free semigroup on $m$-generators $\ell^1(\FS)$.
\end{abstract}



\maketitle


\section{Introduction}




The Cuntz algebras were introduced in~\cite{[Cu]} as new examples of amenable C*-algebras,
and so, from the beginning, they were studied for their homological properties.
The homological properties of Cuntz $\ell^1$-algebras are the central topic of this paper,
and the main results are obtained in Theorem~\ref{cyclic cohomology} and Theorem~\ref{simplicial vanishes}.
The original construction of Cuntz algebras views them as acting on infinite $m$-ary tree structures,
and so the free semigroup on $m$-generators~$\FS$ has always been in clear view.
We consider the Banach algebra $\ell^1(\FS)$ in Section \ref{section_Free_Algebra} where we determine the cyclic and simplicial cohomology of tensor algebras (Theorem \ref{tensor-cyclic} and Theorem \ref{tensor-sim}), a class which includes $\ell^1(\FS)$.
The original construction of Cuntz has been generalized in a number of directions:
firstly to allow for an infinite number of generators,
and then to act on digraph structures more complicated than rooted trees.
These constructions gave rise to the Cuntz-Krieger algebras~\cite{[CK]}.
Later these examples were generalized to higher rank graph structures by Robertson and Steger~\cite{[RS]}.

At the same time as the above process of generalization was going on,
there was a parallel process of simplification and unification.
Many of the original constructions were unified by a construction
which proceeded via intermediate steps which considered semigroups (and corresponding inverse semigroups) acting on sets.
The $\ell^1$ and $L^1$ algebras of these inverse semigroups were then completed to give the desired C*-algebras
which serve as the generalizations of the original Cuntz algebras.
We recommend the reader to consult Paterson's book~\cite{[Pa]}, for details of these constructions.

It is these intermediate $\ell^1$ inverse semigroup algebras, which give rise to the Cuntz C*-algebra, that are the topic of this paper.
As noted above, the Cuntz C*-algebras are amenable and so have trivial cohomology in dual modules,
and so in particular they have trivial simplicial cohomology.
We will show in Theorem~\ref{simplicial vanishes} that the $\ell^1$-algebras corresponding to the Cuntz C*-algebras do not have trivial simplicial cohomology,
in particular they have non-inner derivations into their duals and so cannot be amenable.
Further, we show in Theorem~\ref{simplicial vanishes} that these algebras have trivial simplicial cohomology in dimensions greater than one,
and perhaps more interestingly, that they have the same cyclic cohomology as $\C$, in dimensions greater than~zero.

The examples considered in this paper are generalizations of the case of the bicyclic semigroup algebra, for which the authors obtained more detailed results in~\cite{[GW2010]}.

\section{Background and definitions}

As this paper considers simplicial and cyclic cohomology, rather than Hochschild cohomology with more general coefficients, and as our notation is essentially that of \cite{[CGW]} (and very close to that of~\cite{[GJW]}),
we shall be brief in this section.

Let $\A$ be a Banach algebra and regard $\A'$\/, the dual space of $\A$, as a Banach $\A$-bimodule in the usual way.
For $n\geq 0$, $\SC{n}{\A}$ denotes the space of \dt{$n$-cochains}, $\SZ{n}{\A}$ the subspace of \dt{$n$-cocycles}, and $\SB{n}{\A}$ the subspace of \dt{$n$-coboundaries}.
Note that by convention, $\SC{0}{\A}=\A'$ and
$\SC{n}{\A}=0$ for negative $n$. We shall write
$\HH^n(\A)$ for the quotient space $\SZ{n}{\A}/\SB{n}{\A}$, the $n$th \dt{simplicial cohomology group} of~$\A$.

We specify some notation. For the remainder of this section, $n\ge 0$. An $n$-cochain is a bounded $n$-linear map $T:\A^n\to \A'$, and the \dt{Hochschild coboundary operator} $\dif^n:\SC{n}{\A}\to\SC{n+1}{\A}$ is defined by
\begin{eqnarray*}
(\d^n T)(a_1,\ldots,a_{n+1})(a_{n+2}) &=& T(a_2, a_3,  \ldots, a_{n+1})(a_{n+2}a_1) \\
& & + \sum_{j=1}^{n} (-1)^j T(a_1, a_2, \ldots, a_j a_{j+1}, \ldots, a_{n+1})(a_{n+2}) \\
& & { } + (-1)^{n+1}  T(a_1, \ldots, a_n)(a_{n+1}a_{n+2})
 \end{eqnarray*}
where $a_1,\ldots,a_{n+2} \in \A$.
We shall sometimes omit the superscript and write $\dif$ for $\dif^n$.

For each $n\ge 0$, elements of $\SC{n}{\A}$ may be regarded as bounded linear functionals on the space $\Ho{C}{n}(\A,\A)\defeq A^{\ptp n+1}$, the $(n+1)$-fold completed projective tensor product of $\A$; the coboundary operator $\delta^n:\SC{n}{\A}\to\SC{n+1}{\A}$ is then the adjoint of the bounded linear operator $d^n:\Ho{C}{n+1}(\A,\A)\to\Ho{C}{n}(\A,\A)$, defined on \dt{elementary tensors} $\et=\asubchain 1{n+2}$ by
 \begin{eqnarray*}
d^n (\et) & = &
 a_2\tp \cdots\tp a_{n+1} \tp a_{n+2}a_1 \\
& & + \sum_{j=1}^{n+1} (-1)^j a_1\tp \cdots\tp a_ja_{j+1}\tp \cdots\tp a_{n+2}.
 \end{eqnarray*}
Note that from now on, we shall simply consider any bounded linear map on $\Cc_n(\A,\A)$ as defined once its value on elementary tensors is determined.

We denote by $d^n_i$, $0\le i \le n+1$, the maps $d^n_i : \Cc_{n+1}(\A,\A) \to  \Cc_{n}(\A,\A)$ given by
 \begin{eqnarray*}
d^n_0(\et)&=& \subchain 2 n \otimes a_{n+2} a_1,\\
d^n_i(\et)&=& (-1)^i \subchain 1 i \subchain {i+1} {n+2}, \ i=1, \ldots, n\\
d^n_{n+1}(\et)&=& (-1)^{n+1}  \subchain 1 {n+1} a_{n+2},
\end{eqnarray*}
so that $d^n=\sum_{i=0}^{n+1} d^n_i$.

Simplicial cohomology is closely linked
to \dt{cyclic cohomology}, which we now introduce.
Denote by $\cyc$ the \dt{signed cyclic shift operator} on the simplicial
chain complex:
\begin{equation*}\label{eq:dfn-cyc-shift}
\cyc(\achain) = (-1)^n (a_{n+1}\otimes\asubchain{1}{n}).
\end{equation*}
By abuse of notation, we also write $\cyc$ for the adjoint
operator on the simplicial \emph{cochain} complex.
The $n$-cochain $T$ (in $\SC{n}{\A}$) is called \dt{cyclic} if
$\cyc T = T$, i.e. if $T(\cyc x)=T(x)$ for all $x\in \Cc_n(\A,\A)$, and the linear space of all cyclic $n$-cochains is denoted by~$\CC^n(\A)$.

It is well known that the cyclic cochains $\CC^n(\A)$ form a subcomplex of
$\SC{n}{\A}$, that is $\d\left(\CC^n(\A)\right) \subseteq \CC^{n+1}(\A)$, and
this allows one to define cyclic versions of the spaces defined above, denoted here by $\ZC^n(\A)$, $\BC^n (\A)$ and $\HC^n (\A)$.
Under certain conditions on the algebra $\A$ (see \cite{[He92]}), the cyclic and simplicial cohomology groups are connected via the
Connes-Tzygan long exact sequence
\begin{equation*}\label{eq:CTseq}
\cdots\to \HH^n(\A) \xrightarrow{B} \HC^{n-1}(\A) \xrightarrow{S} \HC^{n+1}(\A) \xrightarrow{I} \HH^{n+1}(\A) \to \cdots
\end{equation*}
where the maps $B$, $S$ and $I$ all behave naturally with respect to algebra homomorphisms. The reader is referred to~\cite{[He92]} for more details.

Finally, as some of the results we obtain can be used in both simplicial and cyclic cohomology, we shall need to move between these two settings. To do so, the following definition will be useful (where $I$ is the identity map).

\begin{Definition}\label{cyclically equivalent} Two chains $x, y\in \Cc_n(\A,\A)$ are \dt{cyclically equivalent} if
\[ x-y \in (I-\cyc)\Cc_n(\A,\A).\]
\end{Definition}

\section{Constructing a contracting homotopy: a general method}

In this section, starting from any map $\rho: \A\to \A\oh\A$ (where $\A$ is a Banach algebra), we define $s^n : \Cc_{n}(\A,\A)\to\Cc_{n+1}(\A,\A)$ and study general properties of $s^n$, $n\ge 0$ and $\sdplusds$, $n\ge 1$. We then describe a general method, mostly useful in cyclic cohomology, to work with $\sdplusds$ in attempting to build a contracting homotopy.

\begin{Definition}\label{def-s} Let $\rho: \A\to \A\oh\A$ and $n\ge 0$. We define
$s^n_i : \Cc_n(\A,\A)\to\Cc_{n+1}(\A,\A)$, for $i=1, \ldots, n+1$, by
\[ s^n_i(\chain) = (-1)^{i} \sum_{j=1}^{\infty}\subchain 1{i-1}\otimes u_j^{(i)}\otimes v_j^{(i)} \otimes \subchain {i+1}{n+1}\, ,\]
where $\rho(a_i)=\sum_{j=1}^{\infty} u_j^{(i)}\otimes v_j^{(i)}$. We shall write this more concisely in the form
\[ s^n_i(\chain) = (-1)^{i} (\subchain 1{i-1}\otimes \rho(a_i) \otimes \subchain {i+1}{n+1})\, .\]
We then define $s^n : \Cc_{n}(\A,\A)\to\Cc_{n+1}(\A,\A)$ by
\begin{equation*}\label{eq:dfn-s}
s^n = \sum_{k=1}^{n+1} s^n_k.
\end{equation*}
\end{Definition}


In cohomology, we need to have the adjoint of $s^n$.
\begin{Definition} Dualising $s^n$, $n\ge 0$, we define $\sigma^n :
\Co{C}{n+1}(\A,\A')\to\Co{C}{n}(\A,\A')$ on $T\in\Co{C}{n+1}(\A,\A')$ by
\begin{equation*}
\sigma^n T(x) = T (s^n(x)),\qquad x\in \Cc_n(\A,\A).
\end{equation*}
\end{Definition}

In our work, it will be useful to consider maps $\rho: \A\to \A\oh\A$ which are left inverses to the {\em multiplication map} $\pi : \A \oh \A \to \A$, defined on elementary tensors by $\pi(a \otimes b)=ab$, for $a$, $b \in \A$. Nevertheless, it will be useful to have results which hold for any $\rho$. The next lemma shows that, for any $\rho$, the map $\sigma^n$ operates in cyclic cohomology: if $T$ is cyclic, then $\sigma^n(T)$ is also cyclic.

\begin{Lemma}\label{s_is_syclic}
Let $\rho: \A\to \A\oh\A$ and let $s^n$ and $s^n_k$, $n\ge 0$, $k=1, \ldots, n+1$ be as above. Then, for all $x\in \Cc_n(\A,\A)$, we have
\begin{enumerate}
  \item $s^n_1(\cyc x) = \cyc^2 s^n_{n+1}(x)$,
  \item $s^n_k(\cyc x) = \cyc s^n_{k-1}(x) \ k=2, \ldots, n+1$, and
  \item $T(s^{n}(\cyc x)=T(s^n x)$ if $T\in \Cc^{n+1}(\A,\A')$ is cyclic.
\end{enumerate}
\end{Lemma}

\begin{proof}
It is sufficient to prove this for elementary tensors $\et=\chain$. We have
\begin{eqnarray*}
s^n_1(\cyc \et) &=& (-1)^n s^n_1(a_{n+1} \otimes \subchain 1 n)  \\
&=& (-1)^{n+1} \rho(a_{n+1}) \otimes \subchain 1 n\\
&=& (-1)^{n+1} \cyc^2 (\subchain 1 n \otimes \rho(a_{n+1}))\\
&=&  \cyc^2 s^n_{n+1}(\et).
\end{eqnarray*}
Also
\begin{eqnarray*}
s^n_k(\cyc \et) &=& (-1)^n s^n_k(a_{n+1} \otimes \subchain 1 n)  \\
&=& (-1)^{n+k} a_{n+1}\otimes \subchain {1}{k-2}\otimes \rho(a_{k-1}) \otimes \subchain {k} n\\
&=& (-1)^{n+k}(-1)^{n+1} \cyc (\subchain {1}{k-2}\otimes \rho(a_{k-1}) \otimes \subchain {k} {n+1})\\
&=& (-1)^{k-1} \cyc (\subchain {1}{k-2}\otimes \rho(a_{k-1}) \otimes \subchain {k} {n+1})\\
&=& \cyc s^n_{k-1}(\et).
\end{eqnarray*}
The last assertion follows immediately from (i) and (ii).
\end{proof}

To cobound a cocycle, we will consider the adjoint of $\sdplusds$ for some well chosen $\rho : \A \to \A \oh \A$.
We regroup in the next lemma some of the key features of $\sdplusds$.

\begin{Lemma}\label{split_everywhere}
For any $\rho : \A \to \A \oh \A$ and $n\ge 1$, and with the notation as above, we have \begin{enumerate}
  \item $\sd k 0 = - \ds 0 {k+1}$ for $k=1, \ldots, n-1$,
  \item $\sd kj = - \ds j{k+1}$ for $1\le j < k \le n$,
  \item $\sd kj = - \ds {j+1}k$ for $1 \le k < j \le n$,
  \item and therefore $\sdplusds$ is given by
  $$\sd n0 + \ds {n+1}{n+1} +\ds 01 + \ds 0{n+1} + \sum_{i=1}^{n} \bigl( \sd ii + \ds ii + \ds i{i+1} + \ds {i+1}i\bigr).$$ \end{enumerate}
  \end{Lemma}

\begin{proof} The proofs are all straightforward. For (i), we see that for $k=1, \ldots, n-1$ and for an elementary tensor $\et=\achain$, we can write
\begin{eqnarray*}
\sd k 0 (\et) &=& s^{n-1}_k (\subchain 2 {n+1} a_1)\\
&=& (-1)^k \subchain 2 k \otimes \rho(a_{k+1})\otimes \subchain {k+2}{n+1}a_1
\end{eqnarray*}
and
\begin{eqnarray*}
\ds 0 {k+1}(\et) &=& (-1)^{k+1} d^{n}_{0}(\subchain 1 k \otimes \rho(a_{k+1}) \otimes \subchain {k+2}{n+1})\\
&=& (-1)^{k+1} \subchain 2 k \otimes \rho(a_{k+1})\otimes \subchain {k+2}{n+1}a_1.
\end{eqnarray*}

For the second statement, for $1\le j < k \le n$, we have
\begin{eqnarray*}
\sd kj (\et) &=& (-1)^j s^{n-1}_k (\subchain 1 j \subchain {j+1} {n+1})\\
            &=& (-1)^{j+k} \subchain 1 j \subchain {j+1}{k}\otimes \rho(a_{k+1})\otimes \subchain {k+2} {n+1}
\end{eqnarray*}
and
\begin{eqnarray*}
\ds j{k+1} (\et) &=& (-1)^{k+1} d^{n}_j (\subchain 1 k \otimes \rho(a_{k+1})\otimes \subchain {k+2} {n+1})\\
            &=& (-1)^{j+k+1}\subchain 1 j \subchain {j+1}{k}\otimes \rho(a_{k+1})\otimes \subchain {k+2} {n+1}.
\end{eqnarray*}
The proof of the third statement is also direct. For the last statement, we have
$$\sdplusds=\sum_{i=0}^n \sum_{k=1}^n \sd ki + \sum_{i=0}^{n+1} \sum_{k=1}^{n+1} \ds ik,$$ in which there are $2(n+1)^2$ terms. Of these terms, exactly $2(n-1)$ cancel by (i), and exactly $2n(n-1)$ cancel by (ii) and (iii), leaving us with $4n+4$ terms.
It is easily checked that these terms are precisely those given in the lemma.
\end{proof}

We can rewrite the terms in Lemma~\ref{split_everywhere} as follows. For $\sum_{i=1}^{n} \bigl( \sd ii + \ds ii + \ds i{i+1} + \ds {i+1}i\bigr)(\et)$, where $\et=\achain$, we have
\begin{eqnarray}
\sd ii(\et)&=& \subchain 1 {i-1} \otimes \rho(a_i a_{i+1})\otimes \subchain {i+2} {n+1},\label{eq:sdds1}\\
\ds ii (\et) &=& \subchain1{i-1} \otimes \pi \rho (a_i) \otimes \subchain {i+1}{n+1}, \label{eq:sdds2}\\
\ds i{i+1} (\et)  &=& - \subchain 1 {i-1} \otimes a_i \rho(a_{i+1})\otimes \subchain {i+2}{n+1},\label{eq:sdds3}\\
\ds {i+1}i (\et) &=& - \subchain 1 {i-1} \otimes \rho(a_i)a_{i+1}\otimes \subchain {i+2}{n+1}.\label{eq:sdds4}
\end{eqnarray}

The four other terms in Lemma~\ref{split_everywhere} give a similar configuration, where a cyclic shift will be useful to help compare two of the terms with the other ones. This is the content of the following lemma.

\begin{Lemma}\label{splitting_the_first}Let $\et=\chain\in\Cc_n(\A,\A)$, $n\ge 1$. Then the following hold.
\begin{eqnarray*}
\sd n0 (\et)&=& (-1)^n \subchain 2 n \otimes \rho(a_{n+1} a_1),\\
\ds {n+1}{n+1}(\et) &=&(-1)^{n} \cyc( \subchain 2{n} \otimes \pi\rho(a_{n+1})\otimes a_1), \\
\ds 01 (\et)  &=& (-1)^{n+1}\cyc( \subchain 2 {n} \otimes a_{n+1} \rho(a_1)),\\
\ds 0{n+1} (\et) &=& (-1)^{n+1} \subchain 2 n \otimes \rho(a_{n+1}) a_1.
\end{eqnarray*}
\end{Lemma}

\begin{proof}
This is clear by direct calculations. \end{proof}

\subsection{A general method} \label{section method}

In cyclic cohomology, when considering the adjoint of $\sdplusds$, it will be sufficient to analyze
$$\rho(ab) + \pi \rho(a)\otimes b - a \rho(b) - \rho(a)b,$$ for appropriate values of $a, b \in \A$.
This follows from lemmas \ref{split_everywhere} and \ref{splitting_the_first} together with the following observations.

For $1\le i \le n$, it follows from
(\ref{eq:sdds1}), (\ref{eq:sdds2}), (\ref{eq:sdds3}) and (\ref{eq:sdds4}) that the value of
$(\sd ii + \ds ii + \ds i{i+1} + \ds {i+1}i)(\et)$, where $\et$ is an elementary tensor,
is easily determined once we have obtained the value on the part of the cochains which may differ in these four terms, namely
\begin{align*}
&{\rm for\ }\sd ii(\et):& +\ \rho(a_i a_{i+1}), &\\
&{\rm for\ } \ds ii(\et):& +\ \pi \rho(a_i) \otimes a_{i+1}, &\\
&{\rm for\ } \ds i{i+1}(\et):& -\ a_i \rho(a_{i+1}),&\\
&{\rm for\ } \ds {i+1}i(\et):& -\ \rho(a_i)a_{i+1}.&
\end{align*}
For instance, if $\rho(a_i a_{i+1})=\ a_i \rho(a_{i+1})$, then $(\sd ii + \ds i{i+1})(\et)=0$.

For $\sd n0 + \ds {n+1}{n+1} +\ds 01 + \ds 0{n+1}$, the analysis is similar in cyclic cohomology by Lemma \ref{splitting_the_first}.
For instance, when any terms cancel in $\rho(ab) + \pi \rho(ab) - a \rho(b) - \rho(a)b$,
then the corresponding terms in  $\sd n0 + \ds {n+1}{n+1} +\ds 01 + \ds 0{n+1}$ cancel {\em up to cyclic equivalence}.

\section{The Cuntz semigroup algebra} \label{s:Cuntz and reduced}

Let $\Cuntz$ ($1\le m \le \infty$) be the abstract semigroup generated by a unit \one, a zero \zero\ and $p_i$, $q_i$, $1 \le i \le m$ such that $$q_i p_j = \delta_{ij} \one.$$
Apart from \one\ and \zero, the elements of $\Cuntz$ can be written uniquely in the form $\palpha \qbeta$ where $\alpha, \beta$ are finite strings of integers $i$ ($1\le i \le m$), one of which could be empty.
In our work, it will be useful to extend this when we have two empty strings, so that we can write $\one = p_{\alpha} q_\alpha$ where $\alpha$ is the empty string.

The product is given by
$$(p_\alpha q_{\beta \gamma})(p_\beta q_{\beta'}) = p_\alpha q_{\beta'\gamma}$$
and
$$(p_\alpha q_{\beta})(p_{\beta\gamma} q_{\beta'}) = p_{\alpha \gamma}q_{\beta'},$$
where $\alpha\beta$ denotes the string formed by the integers in $\alpha$ followed by those in $\beta$.
Note that $q_\beta p_\beta=\one$, $p_\alpha p_\beta = p_{\alpha \beta}$ and
$q_\alpha q_\beta = q_{\beta \alpha}$.

One can think of this semigroup as a path semigroup in the following way. Take the directed graph consisting of $m$ directed loops at a single vertex. Then $\alpha$ is a directed path, simply the path formed by walking along the loops in the order they appear in $\alpha$. Now $p_\alpha$ means to go forward along this path, and $q_\alpha$ to go backwards along this same path (starting from the end and backtracking). The semigroup contains paths which consist of going forward along a path, and then going backward along a path (the same or different). The rules for the product are that if you go backward on a loop and then forward on a loop, the product is zero unless it is the same loop, in which case the product is $\one$.

\begin{remark} The definition and most of the notation is that of Renault as presented in \cite{[Pa]}. To avoid conflicting notation with our $s_k^n$ and $\cyc$, the elements $s_i$ are written here as $p_i$, and the $t_j$ as $q_j$.
\end{remark}
%
%
%

Let $\ell^1(\Cuntz)$ be the semigroup algebra of $\Cuntz$, with convolution product. We denote the point mass at $\palpha \qbeta$ by $\palpha \qbeta$ itself, and similarly $\zero$ denotes point mass at $\zero$, and $\one$ point mass at $\one$, so that product of these elements in $\ell^1(\Cuntz)$ is directly given by the operation in the semigroup.

\begin{Definition} Let $\Cuntz$ be a Cuntz semigroup, and let $\C \zero$ be the ideal generated by $\zero$ in $\ell^1(\Cuntz)$. The \dt{reduced Cuntz algebra} $\AC$ is the Banach algebra given by
$$\AC=\ell^1(\Cuntz)/ \C \zero.$$
\end{Definition}
It follows directly from \cite[Theorem 4.2]{[Ly98]} that $\AC$ and $\ell^1(\Cuntz)$ have the same simplicial and cyclic cohomology, and in the sequel we shall work with $\AC$.

\section{Construction of a contracting homotopy in cyclic cohomology}

\subsection{The first reduction}
For the reduced Cuntz algebra $\AC$, we consider the following map.
\begin{Definition} Let $\rho: \AC \to \AC \oh \AC$ be given by
$$\rho(\palpha\qbeta)=\palpha\otimes \qbeta.$$
\end{Definition}

Let $s^n_k$ and $s^n$  be defined from this map $\rho$ by Definition \ref{def-s}.
As we will show, the map $\sdplusds$ will enable us to move the problem of cobounding a cochain on arbitrary chains to one where the chains have some added specific features. Let us define the appropriate notions.

\begin{Definition} We define:

\begin{enumerate}
\item A chain $\et\in\Cc_n(\AC,\AC)$ is called an \dt{elementary tensor} if it is of the form $\et=\chain$ where each $a_i$ is a point mass (i.e. $a_i=\p i \q i$).
\item We say that \dt{an elementary tensor $\et$ has a transition at $a_i$}, for $1\le i \le n$, if $a_i=\p i \q i$ and $a_{i+1} = \p {i+1} \q {i+1}$ with $\beta_i$ and $\alpha_{i+1}$ non empty.
Such a transition is called \dt{orthogonal} if $a_i a_{i+1} = 0$.
We say that $\et$ has a \dt{(orthogonal) transition at} $a_{n+1}$ if the same holds with $n+1$, $1$ in lieu of $i$, $i+1$.
\end{enumerate}
\end{Definition}

\begin{Theorem}\label{k_transitions} Let $\et=\chain\in\Cc_n(\AC,\AC)$, $n\ge 1$, be an elementary tensor with $k$ transitions, $0\le k \le {n+1}$, and suppose that $l$ of these transitions are orthogonal $(0\le l \le k)$. Then $(\sdplusds)(\et)$ is cyclically equivalent to $$k\et + \sum_{j=1}^{k-l} {\bf y}_j$$
where each elementary tensor ${\bf y}_i\in\Cc_n(\AC,\AC)$ has exactly $k-1$ transitions. \end{Theorem}

\begin{proof} The proof uses the general method outlined in Subsection~\ref{section method}.
We consider $$\rho(ab) + \pi \rho(a)\otimes b - a \rho(b) - \rho(a)b,$$ for $a, b \in \AC$ , and recall that the statement of the theorem requires an identification up to cyclic equivalence.

First note that $\pi \rho(a) = a$ for all $a \in \AC$ , so that the term corresponding to $\pi \rho(a)\otimes b$ is always cyclically equivalent to $\et$. (Note that here, in fact, we have $\ds ii(\et) = \et$ for $1\le i \le n+1$ and do not need cyclic equivalence. However one advantage of this general approach in cyclic cohomology is that we do not need to keep track of this - and we shall not do so in the remainder of this proof as it provides no relevant information.)

We are left to consider $\rho(ab) - a \rho(b) - \rho(a)b$. Writing $a=\pa 1 \qb 1$ and $b=\pa 2 \qb 2$, we have
$$\rho(ab) - a \rho(b) - \rho(a)b = \rho(ab)- \pa 1 \qb 1 \pa 2 \otimes \qb 2  - \pa 1 \otimes \qb 1 \pa 2 \qb 2.$$

\noindent Let us first consider the cases where there is a transition at $a$.
\begin{itemize}
\item If $ab= 0$ (the transition is orthogonal), then $\qb 1 \pa 2 = 0$ and therefore $\rho(ab) - a \rho(b) - \rho(a)b=0$ as each of these terms vanish. So terms corresponding to these add to 0 up to cyclic equivalence.

\item If $a b \ne 0$ (the transition is not orthogonal), then there are two cases to consider.

If $\qb 1 \pa 2 = p_{\alpha'}$ for some $\alpha'$ (i.e. if $\alpha_2 = \beta_1 \alpha'$), then $\rho(ab) - a \rho(b)=0$ and $- \rho(a)b = - \pa 1 \otimes p_{\alpha'} \qb 2$. This last term has no transition, and so we have one term with one fewer transition.

A similar result is true if $\qb 1 \pa{2} = q_{\beta'}$ for some $\beta'$, as we obtain $\pa 1 q_{\beta'} \otimes \qb 2$, which has no transition.

    \end{itemize}

\noindent Let us now consider the case where there is no transition at $a$. This means that $\beta_1$ or $\alpha_2$ is empty. In the first case, $\rho(ab) - a \rho(b) = 0$ and $-\rho(a)b=-a\otimes b$ which cancels out with $\pi \rho(a)\otimes b$. Similarly, when $\alpha_2$ is empty we get $\rho(ab) - \rho(a)b = 0$ and $- a \rho(b) = - a \otimes b$. Thus, in these cases, the total contribution of the terms corresponding to $\rho(ab) + \pi \rho(a)\otimes b - a \rho(b) - \rho(a)b$ is cyclically equivalent to 0.
\end{proof}

For $n\ge 1$, let $P= \sdplusds$ and let $\Phi: \Cc_n(\AC,\AC) \to \Cc_n(\AC,\AC) $ be the bounded linear map given, for an elementary tensor $\et\in \Cc_n(\AC,\AC) $, by
$$\Phi(\et) = \left(\prod_{j=1}^{n+1} (I - \frac{1}{j}P)\right)(\et) = \frac{1}{(n+1)!} \left(\prod_{j=1}^{n+1} (jI - P)\right)(\et).$$

 \begin{Theorem}\label{no_transition_left} For any chain $x\in \Cc_n(\AC,\AC) $, $n\ge 1$, $\Phi(x)$ is cyclically equivalent to a sum of cochains with no transition.
 \end{Theorem}

 \begin{proof} For any elementary tensor $\et\in \Cc_n(\AC,\AC) $ with $k$ transitions, it follows from a repeated application of Theorem \ref{k_transitions} that
 $$(I - \frac{P}{1}) (I - \frac{P}{2}) \cdots (I- \frac{P}{k})(\et)$$
 is a sum of cochains with no transitions. As the $I - \frac{1}{j}P$ are pairwise commuting, we have
 $$\Phi(\et) = \left(\prod_{j=k+1}^{n+1} (I - P/j)\right) (I - \frac{P}{1}) (I - \frac{P}{2}) \cdots (I- \frac{P}{k})(\et).$$
 The result now follows immediately as $P$ does not increase the number of transitions of an elementary tensor, and therefore neither does $I-P/j$ for any $j\ge k+1$.
  \end{proof}

\subsection{The second reduction}\label{section_second_reduction}
In this subsection, we look at elementary tensors $\et\in \Cc_n(\AC,\AC) $ with no transitions.
First, we introduce basic definitions and some notation.
\begin{Definition} For a string of integers $\alpha=i_1 i_2 ... i_n$,
\begin{enumerate}
\item the \dt{length of $\alpha$}, denoted by $\order(\alpha)$, is the number of integers in the string,
\item for $1 \le j\le k \le \order(\alpha)$, we denote by $\alpha_{[j,k]}$ the string $i_j i_{j+1} ... i_k$.
\end{enumerate}
\end{Definition}

\begin{Definition}
For $a = \palpha \qbeta \in \AC$ , the {\em length of $a$}, denoted by $\order(a)$, is defined by $\order(a)=\order(\alpha)+\order(\beta)$.
The {\em length of an elementary tensor $\et=\chain$} is $\order(\et)=\sum_{i=1}^{n+1} \order(a_i)$.
\end{Definition}

\begin{Definition}\label{def-zmap} Let $\bar{\rho}: \AC \to \AC \oh \AC$  be the bounded linear map given, for $a=\palpha\qbeta$ with $\order(\alpha)=N$ and $\order(\beta)=M$, by
\begin{eqnarray*}
\bar{\rho}(a)= \frac{1}{\order(a)}\rho(a),
\end{eqnarray*}
where $\rho: \AC \to \AC \otimes \AC$  is the (unbounded) linear map defined by
\begin{eqnarray*}
\rho(\palpha\qbeta)= \sum_{k=0}^{N-1} p_{\alpha_{[1,k]}} \otimes p_{\alpha_{[k+1,n]}}\qbeta + \sum_{l=1}^{M} \palpha q_{\beta_{[l+1,m]}} \otimes q_{\beta_{[1,l]}},
\end{eqnarray*}
where the term corresponding to $k=0$ is $\one \otimes \palpha \qbeta$, and the term for $l=M$ is $\palpha\otimes \qbeta$. (Note that $\rho(\one)=0$ as both sums are empty in this case.)

We define $\bar{s}_k^n$ from this map $\bar{\rho}$ as given in Definition~\ref{def-s}.
Finally, we define the bounded linear map $\stwo^n : \Cc^{n}(\AC,\AC) \to\Cc^{n+1}(\AC,\AC) $, $n\ge 0$. On an elementary tensor $\et=\chain$ such that $\order(\et)\ne 0$, we set
\[ \stwo^n(\et) = \frac{1}{\order(\et)}\sum_{k=1}^{n+1} \order(a_k) \bar{s}_k^n(\et),\]
and we set $\stwo^n(\et)=0$ if $\order(\et)=0$.
\end{Definition}

As the set up is slightly different from that of subsection \ref{section method}, we need to explain why we can still use that method. We can rewrite $\stwo^n$ on an elementary tensor $\et$ as
$$\stwo^n(\et)=\frac{1}{\order(\et)} \sum_{k=1}^{n+1}s^n_i(\et)$$
where
$$s^n_i(\et) = (-1)^{i} \subchain 1{i-1}\otimes \rho(a_i) \otimes \subchain {i+1}{n+1}.$$
(Note that we can think of $s^n_i: C_n(\AC,\AC)  \to C_{n+1}(\AC,\AC) $, where $C_n(\AC,\AC) $ is the unbounded (or purely algebraic) space of chains, or as a notational device.)
   Rewriting $\stwo$ in this way makes it clear that the conclusion of Lemma \ref{s_is_syclic} and Lemma \ref{split_everywhere} hold on elementary tensors $\et$,
and therefore Lemma \ref{splitting_the_first} holds in this case, again on elementary tensors $\et$.
Thus we can use the general method described in subsection \ref{section method}.

\begin{remark} Note that the map $\stwo^n$ is very close to the map used in [GJW] for $\ell_1(Z_+)$.
\end{remark}

To proceed, we need a lemma which states an essential property of $\rho$.

\begin{Lemma}\label{split is a derivation} Let $a = \pa 1 \q1$ and $b= \p2\qb 2$.
If there is no transition at $a$, then
$\rho(ab) = a \rho(b) + \rho(a)b,$ for $a, b \in \AC$ .
\end{Lemma}

\begin{proof} If there is no transition,
then when we write $a = \pa 1 \q1$ and $b= \p2\qb 2$, either $\q1$ or $\p2$ is $1$. Therefore the statement is equivalent to proving that
$$\rho(\p1 \p2 \qb 2) = \rho(\p1)\p2\qb 2 + \p1\rho(\p2 \qb 2)$$
and
$$\rho(\pa 1 \q1 \q2) = \rho(\pa 1 \q1) \q2 + \pa 1 \q1\rho(\q2).$$
This is clear from the definition. Note that the result is valid if $a$ or $b$ is $\one$. \end{proof}

\begin{Theorem}\label{no_transitions} Let $\et=\chain\in\Cc_n(\AC,\AC) $, $n\ge 1$, be an elementary tensor without transitions.
If $\order(\et)\ge 1$, then $(\sdplusdstwo)(\et)$ is cyclically equivalent to $\et$. \end{Theorem}

\begin{proof} First note that $\order(\et) = \order(d^n_i(\et))$, $i=0,...,n$, when
 $\et$ has no transition: this is a simple and crucial property of chains with no transition.
Therefore we can use our general method of proof, and simply take into account the factor $\frac{1}{\order(\et)}$
in the definition of $\stwo^{n}$ and $\stwo^{n-1}$ after comparing terms.
From Lemma \ref{split is a derivation}, we have
$$\rho(ab) = a \rho(b) + \rho(a)b,$$ for $a, b \in \AC$ , and therefore only the terms corresponding to
$\pi\rho(a)\otimes b = \order(a)(a\otimes b)$ are left: this gives us $\order(\et) \et$ in total.
Now, considering that we divide by $\order(\et)$, we obtain $\et$, as claimed.
\end{proof}

\section{Cohomology of the Cuntz algebra}

We recall that $\AC=\ell^1(\Cuntz)/ \C \zero$ and $\ell^1(\Cuntz)$ have the same cyclic and simplicial cohomology, and we continue to work with $\AC$ . We start by identifying the zero cyclic and simplicial cohomology groups of $\AC$.

\begin{Definition}
A \dt{trace} on a Banach algebra $\A$ is an element $\tau$ of the dual of the algebra, such that
 $\tau(ab)=\tau(ba)$, for all $a$, $b\in \A$.
\end{Definition}
Observe that it follows that the zero cyclic and simplicial cohomology groups are both exactly the space the traces on any algebra.

Let $\phi$ be a cyclic cocycle. In general, we wish to cobound a cyclic cochain $\chi$ such that $\psi-\dif\chi$ vanishes: this will not be possible in even dimensions, where we first need to take off a trace.

\begin{Lemma}\label{Lemma_trace} Let $\A$ be a Banach algebra and $\tau : \A \to \mathbb{C}$ a bounded trace on $\A$. Let $\tau^{(2n)}\in\CC^{{2n}}(\A)$ be defined by $$\tau^{(2n)}(\asubchain 1 {2n+1})=\tau(a_1 a_2 \cdots a_{2n+1}).$$ Then $\tau^{(2n)}$ is a cyclic cocycle.
\end{Lemma}

\begin{proof} The proof is immediate. \end{proof}

\begin{Proposition} \label{removing trace} Let $\phi\in\ZC^{2n}(\AC)$, $n\ge 1$, and let $\lambda = \phi(\one \otimes \cdots \otimes \one)$. Let $\tau : \AC \to \mathbb{C}$ be defined by $\tau(p_\beta q_\beta)=\tau(\one)=\lambda$, zero elsewhere. Then,
\begin{enumerate}
\item $\phi-\tau^{(2n)}\in\CC^{{2n}}(\AC)$ is a cyclic cocycle such that
$$(\phi-\tau^{(2n)})(\one \otimes \cdots \otimes \one)=0;$$
\item if $\lambda\ne 0$, $\tau^{(2n)}$ is not a coboundary.
\end{enumerate}
\end{Proposition}

\begin{proof} To prove the first statement, it is sufficient to show that $\tau$ is a bounded trace (Lemma \ref{Lemma_trace}). This is straightforward. Let an arbitrary $\beta$ be written as $\beta=\alpha\gamma$ for some $\alpha$, $\gamma$. For $\tau$ to be a trace we need $\tau(p_{\alpha}p_{\gamma} q_{\beta'}) = \tau(p_\gamma q_{\beta'} p_\alpha)$ and $\tau(p_{\beta'}q_{\alpha}q_{\gamma} ) = \tau(q_\alpha p_{\beta'} q_\gamma)$. If $\beta=\beta'$, then all of these terms are $\tau (p_\gamma q_\gamma)$, and are therefore equal. If $\beta\ne\beta'$, then $\tau$ is zero on all terms.

The second statement is clear as $d^{2n}(\one \otimes \cdots \otimes \one)=0$.
\end{proof}

\begin{Definition}
Let $\phi\in\ZC^n(\ell^1(S))$. We say that $\phi$ is \dt{$\one$-normalized} if it vanishes on $\et = \one \otimes \cdots \otimes \one$.
\end{Definition}

\begin{Theorem} \label{cobounding one-normalized}Let $n\ge 1$, and let $\phi\in\ZC^{n}(A)$ be a $\one$-normalized cyclic $n$-cocycle. Then $\phi$ is a cyclic coboundary.
\end{Theorem}

\begin{proof} First, let $\et=\chain\in\Cc_n(\AC,\AC) $ be a chain with no transitions such that $\order(\et)\ge 1$.
Then $\phi(I-\sdplusdstwo)(\et)=0$ (Theorem \ref{no_transitions}). Let $\psi_1\in\CC^{n-1}$ be given by $\psi_1=\phi \circ s^{n-1}$, and let $\phi_1 := (\phi - \delta\psi_1)$. Then $\phi_1(\et)=0$, and we note that $\phi_1(\one \otimes \cdots \otimes \one)= 0$.

For an arbitrary $\et=\chain\in\Cc_n(\AC,\AC) $, we deduce from Theorem \ref{no_transition_left} that
$\phi_1(\Phi(\et))=0$. It is standard that $\Phi(\et)=I-({\tilde s}^{n-1}d^{n-1}+d^{n}{\tilde s}^n)$ for some map $\tilde s$, and therefore $\psi_2\in\CC^{n-1}$ given by $\psi_2=\phi_1 \circ \tilde s^{n-1}$ is such that $(\phi_1 - \delta(\psi_2))(\et)=0$. As all maps considered are clearly bounded, we deduce that $(\phi-\delta(\psi_1+\psi_2))(x) =0$ for all $x\in\Cc_n(\AC,\AC) $.
\end{proof}

The next theorem is an immediate consequence of Proposition \ref{removing trace} and Theorem \ref{cobounding one-normalized}.

\begin{Theorem} \label{cyclic cohomology} The cyclic cohomology of the Cuntz algebra $\ell^1(\Cuntz)$ and of the reduced Cuntz algebra $\AC$  is zero in odd dimensions, and is $\mathbb{C}$ in even dimensions greater than 0.
\end{Theorem}

\begin{proof} In odd degrees, a cyclic cocycle $\phi\in\ZC^{2n}(A)$ is always such that $\phi(\one \otimes \cdots \otimes \one)=0$, and is therefore a coboundary (Theorem~\ref{cobounding one-normalized}). In even degrees, the results immediately follows from Proposition \ref{removing trace} and Theorem \ref{cobounding one-normalized}.
\end{proof}

\begin{Theorem}\label{simplicial vanishes}
The simplicial cohomology of the Cuntz algebra $\ell^1(\Cuntz)$ and of the reduced Cuntz algebra $\AC$  is
zero in degrees $2$ and above, and the first simplicial cohomology group is isomorphic to the space of traces vanishing at~$\one$.
\end{Theorem}

\begin{proof}
As $\AC$  is unital, by \cite[Theorem~16]{[He92]}, the Connes-Tzygan sequence for $\AC$  exists.
Since the cyclic cohomology of $\AC$  vanishes in all odd degrees, the long exact sequence breaks up to give exact sequences
\[ 0\to \HH^{2n-1}(\A) \xrightarrow{B} \HC^{2n-2}(\A) \xrightarrow{S} \HC^{2n}(\A) \to \HH^{2n}(\A) \to 0 \]
for all $n\geq 1$\/.
 By Theorem~\ref{cyclic cohomology}, every cyclic $2n$-cocycle is cohomologous to one of the form $\tau^{(2n)}$ (a one-dimensional space -- see Proposition \ref{removing trace}) and we have $\tau^{(2n)}=\lambda_n^{-1}S\tau^{(2n-2)}$ for some constant $\lambda_n$ (see \cite{[C2010]} for a proof). Thus $S: \HC^{2n-2}(\A)\to \HC^{2n}(\A)$ is surjective for all $n\ge 1$, and bijective for $n\ge 2$ (as both spaces are one-dimensional, and the map $S$ acts as the identity on cochains given by traces). Therefore $\HH^{2n}(\A)=0$ for all $n\ge 1$ and $\HH^{2n-1}(\A)=0$ for all $n\geq 1$\/. The first cohomology group $\HH^1(\A)$
is isomorphic to the kernel of $S$ which can be easily identified as the space of traces vanishing at~$\one$
 \end{proof}
 
 We will give a precise description of the space of traces on $\AC$ in Lemma \ref{hh0}.

\section{Free Algebra}\label{section_Free_Algebra}

The Cuntz semigroup $\Cuntz$ contains two natural copies of the free semigroup on $m$ generators, which we will denote $\FS$. One of these semigroups is generated by $p_1,\ldots, p_m$ and the other by $q_1,\ldots, q_m$.
These inclusions are similar to the inclusions of the two copies of $\Z_+$,
in the case of the bicyclic semigroup, which was considered in~\cite{[GW2010]}.
In that paper, both of the semigroups, which in that case were isomorphic to $\Z_+$, gave rise to some of the cohomology.

In the case of the bicyclic semigroup algebra we have that the simplicial cohomology is trivial for dimensions strictly greater than~$1$,
as is the case for $\ell^1(\Z_+)$, see~\cite{[GJW]}.
Given this, it is natural to ask whether $\ell^1(\FS)$ also has trivial simplicial cohomology groups for dimensions greater than~$1$.
We will see below in Theorem~\ref{tensor-sim} that this is the case. This theorem generalizes the result for $\ell^1(Z_+)$, given in~\cite{[GJW]}.
In fact the proof is very similar to that given in that paper, in that it goes via proving that the cyclic cohomology of
$\ell^1(\FS)$ is the same as that of $\C$ in all but dimension zero, and then deduces the result for simplicial cohomology by using the Connes-Tzygan long exact sequence.

The result on cyclic cohomology holds in slightly more generality than for~$\ell^1(\FS)$.
It holds for any free $\ell^1$-algebra over a Banach space~$V$.
This free algebra is usually called the {\em tensor algebra} for the Banach space for reasons which are made clear by the following definition.

\begin{Definition}\label{tensor-alg}
Let $V$ be a complex Banach space. We define the \dt{tensor algebra over $V$}, denoted $T(V)$, to be the Banach space given by
 $$ T(V) = \ell^1-\bigoplus_{k=0}^\infty \hat\bigotimes^k V. $$
 The tensor products, $\hat\bigotimes^k V$, are the (completed) projective tensor product with the usual norm, with the convention $\bigotimes^0 V := \mathbb{C}$. The direct sum is the $\ell^1$-direct sum.

 The Banach algebra product is given by the linear extension of the isometric isomorphisms:
  $$ \mu:\hat\bigotimes^k V \quad \hat\otimes \quad \hat\bigotimes^l V  \to \hat\bigotimes^{k+l} V,$$
  which is defined by concatenation on elementary tensors:
  $$ \mu( (v_1 \otimes \cdots \otimes v_k) \otimes (w_1 \otimes \cdots \otimes v_l) ) = (v_1 \otimes \cdots \otimes v_k \otimes w_1 \otimes \cdots \otimes v_l).$$
  Elements of $\bigotimes^0 V = \mathbb{C}$ will be denoted by $\lambda \one$: with this notation, $\one$ is the unit of the Banach algebra $T(V)$.
\end{Definition}

Example: If we take $V$ to be $\ell^1_m$, the $m$-dimensional $\ell^1$-space, with standard basis $e_1, \ldots, e_m$,
then the tensor algebra~$T(\ell^1_m)$, is isometrically isomorphic to the free semigroup algebra~$\ell^1(\FS)$,
which we take to be generated by $p_1,\ldots, p_m$, via the unital homomorphism which takes $e_j$ to $p_j$ for ($1\le j \le m$).
For example $ (e_1 + 4 e_2)\otimes ( e_3 - e_1) $ maps to $ p_1 p_3 - p_1^2 + 4 p_2 p_3 - 4 p_2 p_1$.

\begin{Theorem}\label{tensor-cyclic} The cyclic cohomology of the tensor algebra~$T(V)$ of a Banach space~$V$ is zero, in odd dimensions, and is equal to $\C$ in even dimensions greater than~$0$.
\end{Theorem}

\begin{proof} We define a bounded linear map $\sfg:\Cc_{n}(T(V),T(V))\to\Cc_{n+1}(T(V),T(V))$ in a way which is very close to that of Section \ref{section_second_reduction}, where we defined the map $\stwo^n$ to deal with elementary tensors with no transitions for the reduced Cuntz algebra.

We introduce some notation. For an elementary tensor $\etv =v_1 \otimes \cdots \otimes v_N \in \hat\bigotimes^N V$, we denote by $o(\etv)$ the order $N$ of $\etv$, and by $\etv_{[j,k]}$, $0<j\le k\le o(\etv)$, the element $v_j \otimes \cdots \otimes v_k \in \hat\bigotimes^{k-j+1} V$. If $j=k$, then $\etv_{[j,k]}:=\one$.

We define a bounded linear map $\rhov: T(V) \to T(V) \oh T(V)$.
First, for an elementary tensor $\etv =v_1 \otimes \cdots \otimes v_N \in
 \hat\bigotimes^N V$ with $N\ge 1$, let
 $$ \rho(\etv) = \sum_{k=0}^{N-1}  \etv_{[0,k]} \otimes \etv_{[k+1,N]},$$
 where the term corresponding to $k=0$ is $\one \otimes \etv$. (Note that $\rho(\one)=0$ as the sum is empty.)
Then $\rhov$ is given by $$\rhov(\etv)=\frac{1}{N} \rho(\etv)$$
 on elementary tensors, extended by linearity and boundedness to the whole of $T(V)$.
For $N=0$, both $\rho$ and $\rhov$ are defined to be the zero map.

In a way analogous to Definition \ref{def-zmap}, for $n\ge 0$, we define $\bar{s}^n_k$, $k=1, \ldots, n+1$,
from $\rhov$ as given by Definition \ref{def-s}, i.e.
\[ \bar{s}^n_i(\et) = (-1)^{i} \sum_{j=1}^{\infty} \subvchain 1{i-1}\otimes \rhov(\etv_i) \otimes \subvchain {i+1}{n+1}\, .\]
We finally define $\sfg^n : \Cc_{n}(T(V),T(V))\to\Cc_{n+1}(T(V),T(V))$ by
\begin{equation}\label{eq:dfn-sfg}
\sfg^n(\et) = \sum_{k=1}^{n+1} \frac{o(\etv_k)}{o(\et)} s^n_k,
\end{equation}
where $\et\in\Cc_{n}(T(V),T(V))$ is an elementary tensor (i.e. $\et=\subvchain{1}{n+1}$ where each $\etv_i$ is itself an elementary tensor in $\bigotimes^{o(\etv_i)} V)$. Note that $o(\et)$, the order of the elementary tensor, is $o(\et)=\sum_{i=1}^{n+1} o(\etv_i)$.
Again, we define the map to be zero on order zero elements.
We are now in a situation which is very close to that of Definition \ref{def-zmap}.
We can verify that statements analogous to Lemma \ref{split is a derivation} and Theorem \ref{no_transitions} hold.
\begin{enumerate}
\item $\rho(ab) = a \rho(b) + \rho(a)b,$ for $a, b \in T(V)$.
\item For an elementary tensor $\et=\subvchain{1}{n+1}\in\Cc_n(T(V),T(V))$, $n\ge 1$, $(\sdplusdsv)(\et)$ is cyclically equivalent to $\et$.
\end{enumerate}
The proof of (i) is easy, and as we have $\order(\et) = \order(d^n_i(\et))$, $i=0,...,n+1$, the proof of (ii) is essentially that of
Theorem \ref{no_transitions}.

To complete the proof, we only need to repeat the end of the argument for the Cuntz algebra. We use Proposition \ref{removing trace} for the trace $\tau$ defined by $\tau(\one)=\lambda$, $\tau$ zero elsewhere, do as in Theorem \ref{cobounding one-normalized} to show that reduced cocycles cobound (the case at hand is simpler in fact), and we have the result as in Theorem \ref{cyclic cohomology}.
\end{proof}

Before we can use the Connes-Tzygan long exact sequence to identify the simplicial cohomology of $T(V)$, we first need to identify the space of traces.
Let us denote by $\tau:\hat\bigotimes^k V \to \hat\bigotimes^k V$, the maps which satisfies
$\tau( v_1 \otimes \cdots \otimes v_k) = v_k \otimes v_1 \otimes \cdots \otimes v_{k-1}$.

We will use the same symbol, $\tau$, to denote the isometric isomorphisms which this induces from $T(V)$ to itself.

We denote the $\left(\hat\bigotimes^k V\right)^\tau$ the subspace of $\hat\bigotimes^k V$, which is invariant under the action of~$\tau$.
We call the direct sum
 $$  \ell^1-\bigoplus_{k=0}^\infty \left(\hat\bigotimes^k V\right)^\tau$$
 the {\em space of invariants}.

It is clear that a trace on $T(V)$ will be invariant under $\tau$, and similarly that any functional which is invariant under $\tau$ defines a trace on $T(V)$.
Thus we have:

\begin{Lemma} \label{invariants} Let $T(V)$ be the tensor algebra of the Banach space~$V$, then $\HC^0(T(V))$ is isomorphic to the dual of the space of invariants, i.e.,
 $$ \HC^0(T(V)) = \left(\ell^1-\bigoplus_{k=0}^\infty \left(\hat\bigotimes^k V\right)^\tau \right)'.$$
\end{Lemma}

We can now deduce the simplicial homology of $T(V)$, in the same manner as we deduce this for $\ell^1(\Cuntz)$ in Theorem~\ref{simplicial vanishes}.

\begin{Theorem} \label{tensor-sim}
Let $V$ be a Banach space, then the simplicial cohomology of the tensor algebra $T(V)$ is zero in degrees $2$ and above.
In dimension~$0$ it is the space of traces on $T(V)$ and in dimension~$1$, we have $\HH^1(T(V))$ is isomorphic to the spaces of traces which vanish on~$\one \in T(V)$.
\end{Theorem}

\begin{proof} The key observation is that just as in Theorem~\ref{simplicial vanishes}, we know that the cyclic cohomology is zero, in odd dimensions, and is one dimensional and given by a trace, in even dimensions (greater than~$0$). (In Theorem~\ref{simplicial vanishes}, we knew these facts to hold for~$\ell^1(\Cuntz)$.)
The proof remains valid for $T(V)$, based on the same facts, and we obtain the result.
\end{proof}

Now that we have identified the space of traces on the $T(\ell^1_m)$ in~Lemma~\ref{invariants}, and hence also on the isomorphic algebra~$\ell^1(\FS)$, we can quickly describe the traces on $\ell^1(\Cuntz)$.
We will denote the two copies of the free semigroup algebra in $\ell^1(\Cuntz)$,
which are generated by $p_1, \ldots, p_m$ (and the other by $q_1, \ldots, q_m$), by $\ell^1(\F_p)$, (respectively $\ell^1(\F_q)$).

Note that each of these algebras is isomorphic to $\ell^1(\FS)$, for which we know the space of traces.

\begin{Lemma}\label{hh0}
The space of traces on the
reduced Cuntz algebra $\AC$  is isomorphic to the subspace of the direct sum of the traces on $\ell^1(\F_p)$ and $\ell^1(\F_q)$, which agree on the unit elements.
\end{Lemma}

\begin{proof} Clearly any trace $\tau$ on $\AC$, induces a trace on each of $\ell^1(\F_p)$ and $\ell^1(\F_q)$ by restriction. Further, the two restrictions agree on the corresponding unit elements.
This defines the map from $\HC^0(\AC)$ to $\HC^0(\ell^1(\F_p))\oplus \HC^0(\ell^1(\F_q))$.

Next we see that this map is injective. Note that the trace $\tau$ has the same value on $a=\palpha \qbeta$ and on $\qbeta \palpha$. However, such a product is equal to either $q_{\beta'}$ for some non empty string $\beta'$, or $p_{\alpha'}$ for some non empty string $\alpha'$, or it is $\one$ or is zero. In any case the value is determined by the value of one of the restrictions, and so the map in injective.

We must now show that the map is surjective. Namely, that any pair of traces $\tau_p \in \HC^0(\ell^1(\F_p))$ and $\tau_q \in \HC^0(\ell^1(\F_q))$ such that $\tau_p(\one) = \tau_q(\one)$, determine a trace on $\AC$.
It is clear from the argument above that the values of such a $\tau$ are determined by 
$ \tau(\palpha \qbeta) = \tau(\qbeta \palpha)$, which is given by either $\tau_p$ or $\tau_q$.

We must now verify that~$\tau$ defined in this manner is indeed a trace. When checking that $\tau(ab) = \tau(ba)$, each of the products $ab$ and $ba$ may give rise to a canceling of $q$'s and $p$'s, and each of the resulting product may (a priori) require a different value, taken from either $\tau_p$ or $\tau_q$. We need to ensure that the values always agree.

There are several cases to consider. To begin, we observe that $\tau(p_i c q_j) = \delta_{ij} \tau(c)$. We proceed by induction on the total length of the words $a$ and $b$. Consider the case $a = p_i a'$ and $b = b' q_j $ and assume that the results holds for the shorter words $a'$, $b'$. We have 
 $$\tau(  a b ) = \tau( p_i a' b' q_j ) = \delta_{ij} \tau( a' b') = \delta_{ij} \tau( b' a') = \tau( b' q_j p_i a') = \tau(b a).$$
 
Interchanging $a$ and $b$ shows that to base our induction we only need to consider the cases when $a$ and $b$ consist only of $p_i$'s or $q_j$'s, and for these the result is clear. Hence $\tau$ is a trace.
%
\end{proof}

\section{Conclusion}
The results of this paper should be considered to be a first step in the analysis of the
$\ell^1$ and $L^1$ algebras associated with the various generalizations of Cuntz algebras and semigroups.
The next step should be the Cuntz-Krieger algebras and other digraph and quiver type algebras.
This should also include examples like the free semigroup algebras considered here in Theorem~\ref{tensor-sim},
which are not inverse semigroups.
In pure algebra many of these algebras have short resolutions which quickly show that all of the homology of
algebras is supported in low degrees.
Such resolutions are not available for the $\ell^1$-completions of these algebras as is evidenced by the fact that
they do not have all of the cohomology supported in the same low degrees as the algebraic analogues would suggest.
For example: the free semigroups on a single generator $Z_+$ gives rise to the algebra of polynomials in a single variable,
which has dimension~$1$ in pure homology.
However, it is known that the completion of this algebra~$\ell^1(Z_+)$, has non trivial cohomology in dimension~$2$, as is shown in~\cite{[DaDu]}.

{\em Acknowledgment.} The first author thanks the University of Newcastle for its kind hospitality while this paper was being written, and gratefully acknowledges the support of the Natural Sciences and Engineering Research Council of Canada.

\bigskip

\end{document}